\documentclass[a4paper, 11pt]{article}
\usepackage{amsmath, amssymb, amsthm, enumitem, hyperref, mathtools, mathrsfs, newtxtext, newtxmath, tikz-cd, titlesec, url}
\usepackage[font=small, width=0.75\textwidth]{caption}
\usepackage[margin=30truemm]{geometry}

\theoremstyle{definition}
\newtheorem{dfn}{Definition}[section]
\newtheorem{eg}[dfn]{Example}
\newtheorem{ntn}[dfn]{Notation}
\newtheorem{rmk}[dfn]{Remark}
\theoremstyle{plain}
\newtheorem{cor}[dfn]{Corollary}
\newtheorem{lem}[dfn]{Lemma}
\newtheorem{prop}[dfn]{Proposition}
\newtheorem{thm}[dfn]{Theorem}

\newcommand{\category}[1]{\mathbf{#1}}
\newcommand{\categoryMod}{\mathbf{\mathchar`-Mod}}
\newcommand{\set}[1]{\mathbb{#1}}

\newcommand{\closedinterval}{\mathtt{I}}
\renewcommand{\mod}{\mathop{\mathrm{mod}}}
\DeclareMathOperator{\sgn}{sgn}

\newcommand{\demph}[1]{\textbf{#1}}
\newcommand{\affiliation}[1]{\textnormal{\small #1}}
\newcommand{\email}[1]{\texttt{\small #1}}

\numberwithin{equation}{section}
\setlist{itemsep=0ex}
\titleformat*{\section}{\large\bfseries}

\title{A note on symmetric midpoint algebras, or scales}
\author{
  Wataru Tanizaki \\
  \affiliation{Research Institute for Mathematical Sciences, Kyoto University, Japan} \\
  \email{tanizaki@kurims.kyoto-u.ac.jp}
}
\date{}

\begin{document}
\maketitle

\begin{abstract}
In this paper, we investigate some basic properties of algebraic structures called symmetric midpoint algebras, for which we borrow the term “scales” from Peter Freyd. Scales are broadly divided into cancellative ones and non-cancellative ones; the former are precisely subscales of modules over the ring of dyadic rational numbers, and include free scales; the latter can be turned cancellative by taking quotient scales. Furthermore, we can define tensor products of scales similarly to those of abelian groups, with respect to which the category of scales is a symmetric monoidal closed category. We shall call its monoid objects “scalic rings” after the fact that rings are monoid objects in the category of abelian groups.
\end{abstract}

\begin{figure}[h]
  \centering
  \begin{tikzcd}[column sep=1.5em, row sep=2ex]
    & (\S\ref{sec:cancellative scales and their enveloping dy-modules}) & & (\S\ref{sec:congruences and cancellativation}) & & (\S\ref{sec:free scales}) \\
    \set D\categoryMod \arrow["\text{inclusion}"', hook, rr, shift right=2] & & \category{Scl}_c \arrow["\text{inclusion}"', hook, rr, shift right=2] \arrow["\bot", "\text{envelopment}"', ll, shift right=2] & & \category{Scl} \arrow["\text{forgetful}"', rr, shift right=2] \arrow["\bot", "\text{cancellativation}"', ll, shift right=2] & & \category{Set} \arrow["\bot", "\text{free}"', ll, shift right=2]
  \end{tikzcd}
  \caption{The relationship among the categories $\category{Set}$ (of sets), $\category{Scl}$ (of scales), $\category{Scl}_c$ (of cancellative scales), and $\set D\categoryMod$ (of modules over the ring $\set D$ of dyadic rational numbers).}
\end{figure}
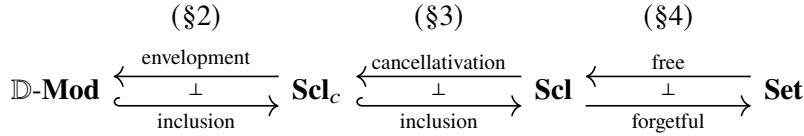

\section{The definition of symmetric midpoint algebras, or scales}

A \emph{scale} is, originally, an algebraic structure (hereinafter, simply referred to as an “algebra”) that \nobreak{Peter~Freyd} introduced in his article \emph{Algebraic real analysis}~\cite{Fre08}, which is defined as a \emph{symmetric midpoint algebra} together with extra operations. Although scales are the main topic of that article, symmetric midpoint algebras are in themselves interesting because they have many similarities with, and at the same time, many differences from abelian groups. The name “symmetric midpoint algebra” is, however, so long that it is inconvenient for naming related concepts. Thus, in this paper, \emph{we borrow the term “scale” to mean a symmetric midpoint algebra itself}.

First, we present the definition of scales that Freyd originally defined, which we shall refer to as “Freyd scales”:

\begin{dfn}[{\cite[Section 2]{Fre08}}]
  A \demph{Freyd scale} is an algebra $(A,\mid,{}^\bullet,{}^\wedge,\top)$ of type $\langle 2,1,1,0\rangle$, consisting of the following fundamental operations:\begin{itemize}
    \item a binary operation $\mid$ called \demph{midpointing};
    \item a unary operation $^\bullet$ called \demph{dotting};
    \item a unary operation $^\wedge$ called \demph{$\top$-zooming};
    \item a nullary operation (that is, a constant) $\top$ called the \demph{top},
  \end{itemize}satisfying the following axioms for all $a,b,c,d\in A$:\begin{itemize}
    \item (idempotence) $a\mid a=a$;
    \item (commutativity) $a\mid b=b\mid a$;
    \item (mediality) $(a\mid b)\mid(c\mid d)=(a\mid c)\mid(b\mid d)$;
    \item $a^\bullet\mid a=\odot$;
    \item $(\top\mid a)^\wedge=a=(\bot\mid a)^\vee$;
    \item $(a\mid b)^\wedge=(a^\vee\mid b^\wedge)^\wedge\mid(a^\wedge\mid b^\vee)^\wedge$,
  \end{itemize}where the following shorthands are used: $a^\vee:=((a^\bullet)^\wedge)^\bullet$; $\bot:=\top^\bullet$; $\odot:=\bot\mid\top$.
\end{dfn}

\begin{eg}
  \label{eg:Freyd scales}
  For any real numbers $a$ and $b$ such that $a\le b$, the closed interval $[a,b]$ is a Freyd scale under the fundamental operations defined by for all $x,y\in[a,b]$:\[
    x\mid y := \frac{x+y}{2}; \quad x^\bullet := a+b-x; \quad x^\wedge := \max\{2x-b,a\}; \quad \top := b.
  \]
\end{eg}

The definition of symmetric midpoint algebras by Freyd can be rephrased as follows (in which we change the operation symbols and names for this paper):

\begin{dfn}[{\cite[Section 3, Footnote 20]{Fre08}}]
  A \demph{scale} (or \demph{symmetric midpoint algebra}) is an algebra $(A,\oplus,-,0)$ of type $\langle 2,1,0\rangle$, consisting of the following fundamental operations:\begin{itemize}
    \item a binary operation $\oplus$ called \demph{midpointing};
    \item a unary operation $-$ called \demph{negation};
    \item a nullary operation $0$ (zero),
  \end{itemize}satisfying the following axioms for all $a,b,c,d\in A$:\begin{enumerate}[label=\text{(S\arabic*)}]
    \item (idempotence) $a\oplus a=a$; \label{scale axiom:idempotence}
    \item (commutativity) $a\oplus b=b\oplus a$; \label{scale axiom:commutativity}
    \item (mediality) $(a\oplus b)\oplus(c\oplus d)=(a\oplus c)\oplus(b\oplus d)$; \label{scale axiom:mediality}
    \item $a\oplus(-a)=0=(-a)\oplus a$; \label{scale axiom:negation}
    \item $-(a\oplus b)=(-a)\oplus(-b)$; \label{scale axiom:-(a+b)=(-a)+(-b)}
    \item $-(-a)=a$. \label{scale axiom:-(-a)=a}
  \end{enumerate}
\end{dfn}

\begin{prop}[self-distributivity]
  Let $A$ be a scale. Then for all $a,b,c,d\in A$,\[
    (a\oplus b)\oplus c = (a\oplus c)\oplus(b\oplus c) \quad \text{and} \quad b\oplus(c\oplus d) = (b\oplus c)\oplus(b\oplus d).
  \]
\end{prop}
\begin{proof}
  Putting $c=d$ (and respectively, $a=b$) in axiom \ref{scale axiom:mediality} and then applying \ref{scale axiom:idempotence}, we have the first (and respectively, second) equation.
\end{proof}

\begin{prop}
  \label{prop:-0=0}
  Let $A$ be a scale, and $a\in A$. Then $-a=a$ if and only if $a=0$. (That is, $0$ is the only element such that $-0=0$.)
\end{prop}
\begin{proof}
  If $a=0$, then\[
    -0 \stackrel{\ref{scale axiom:negation}}{=} -(0\oplus(-0)) \stackrel{\ref{scale axiom:-(a+b)=(-a)+(-b)}}{=} (-0)\oplus(-(-0)) \stackrel{\ref{scale axiom:-(-a)=a}}{=} (-0)\oplus 0 \stackrel{\ref{scale axiom:negation}}{=} 0.
  \]Conversely, if $-a=a$, then\[
    a \stackrel{\ref{scale axiom:idempotence}}{=} a\oplus a = a\oplus(-a)\stackrel{\ref{scale axiom:negation}}{=} 0.
  \]Hence $-a=a$ if and only if $a=0$.
\end{proof}

We define homomorphisms, subalgebras, and direct products of scales similarly to those of other types of algebras.

\begin{dfn}
  \label{dfn:scale homomorphisms}
  Let $A$ and $B$ be a scale. A map $f\colon A\to B$ is a \demph{scale homomorphism} if:\begin{itemize}
    \item (preservation of $\oplus$) $f(a\oplus b)=f(a)\oplus f(b)$ for all $a,b\in A$;
    \item (preservation of $-$) $f(-a)=-f(a)$ for all $a\in A$;
    \item (preservation of $0$) $f(0)=0$.
  \end{itemize}
\end{dfn}

\begin{dfn}
  \label{dfn:subscales}
  A \demph{subscale} of a scale $A$ is a subset $S\subset A$ such that:\begin{itemize}
    \item (closure under $\oplus$) $a\oplus b\in S$ for all $a,b\in S$;
    \item (closure under $-$) $-a\in S$ for all $a\in S$;
    \item $0\in S$.
  \end{itemize}
\end{dfn}

\begin{rmk}
  \label{rmk:redundancy of preservation of 0}
  In Definition \ref{dfn:scale homomorphisms}, the condition that $f(0)=0$ is in fact redundant: if $f$ preserves negation, then we have $f(-0)=-f(0)$, and $f(0)=0$ by applying Proposition \ref{prop:-0=0} twice (or more precisely, the “if” part first and the “only if” part second). In Definition \ref{dfn:subscales}, the condition that $0\in S$ is by axiom \ref{scale axiom:negation} redundant \emph{if $S$ is not empty}. To prove that $S$ has an element, however, it is the easiest to check that $0\in S$ in many cases.
\end{rmk}

\begin{dfn}
  The \demph{direct product} of a family $\{A_i\}_{i\in I}$ of scales is the Cartesian product $\prod_{i\in I}A_i$ together with the pointwise fundamental operations; that is, we define:\begin{itemize}
    \item $(a_i)_{i\in I}\oplus(b_i)_{i\in I}:=(a_i\oplus b_i)_{i\in I}$ for all $(a_i)_{i\in I},(b_i)_{i\in I}\in\prod_{i\in I}A_i$;
    \item $-(a_i)_{i\in I}:=(-a_i)_{i\in I}$ for all $(a_i)_{i\in I}\in\prod_{i\in I}A_i$;
    \item $0:=(0)_{i\in I}$.
  \end{itemize}
\end{dfn}

We write $\category{Scl}$ for the category whose objects are scales and morphisms are scale homomorphisms. Since scale axioms \ref{scale axiom:idempotence}-\ref{scale axiom:-(-a)=a} are written only in equations in terms of $\oplus$, $-$, and $0$, it easily follows that the category $\category{Scl}$ (or simply, the class of scales) is a \emph{variety}, which means that it is closed under taking homomorphic images, subalgebras, and direct products in the category of algebras of type $\langle 2,1,0\rangle$. (For details on varieties, see \cite{Ber11} or \cite{BS12} for instance.)

\begin{eg} \
  \label{eg:cancellative scales}
  \begin{enumerate}[label=(\arabic*)]
    \item Let $A$ be a module over the ring of dyadic rational numbers (that is, integers divided by powers of two). Freyd used the symbol $\set D$ for this ring, and referred to modules over it as \demph{dy-modules}. Then $A$ is a scale under the midpointing $\oplus$ defined by\[
      a\oplus b := \frac{a+b}{2}
    \]for all $a,b\in A$, and the negation and zero inherited from the module structure. When we refer to a dy-module as a scale, we always mean this structure. Note that since the ring $\set D$ is generated by a single element $1/2$, a dy-module can be considered as an abelian group with the following \emph{property} (rather than extra \emph{operations}, that is, multiplication by each scalar):\begin{quote}
      for every element $a$, there is a unique element $a/2$ such that $a/2+a/2=a$.
    \end{quote}Such abelian groups includes:\begin{itemize}
      \item the additive groups $\set D$ of dyadic rational numbers, $\set Q$ of rational numbers, $\set R$ of real numbers, and $\set C$ of complex numbers;
      \item all abelian groups of odd order, and in particular, the trivial group $\{0\}$.
    \end{itemize}
    \item For every real number $r>0$, the closed interval $[-r,r]$ is a subscale of $\set R$, and the closed disk $\{z\in\set C\mid|z|\le r\}$ is a subscale of $\set C$. These are examples of subscales of dy-modules which are not sub-dy-modules.
  \end{enumerate}
\end{eg}

It is well known that for abelian groups $A$ and $B$, the set $\category{Ab}(A,B)$ of all group homomorphisms from $A$ to $B$ is an abelian group under the pointwise addition. This is also true for scales:\begin{prop}
  For all scales $A$ and $B$, the set $\category{Scl}(A,B)$ of all scale homomorphisms from $A$ to $B$ is a subscale of $B^A$ (the scale of all maps from $A$ to $B$ under the pointwise fundamental operations).
\end{prop}
\begin{proof}
  If maps $f,g\colon A\to B$ are scale homomorphisms, then it easily follows from the axioms and proposition shown in Table \ref{tab:scale of scale homomorphisms} that so are $f\oplus g$, $-f$, and $0$. (Recall from Remark \ref{rmk:redundancy of preservation of 0} that if a map from $A$ to $B$ preserves negation, then it automatically preserves zero.) \begin{table}[h]
    \centering
    \caption{The axioms and proposition by which $f\oplus g$, $-f$, and $0$ are scale homomorphisms.}
    \label{tab:scale of scale homomorphisms}
    \begin{tabular}{lccc} \hline
      & $f\oplus g$ & $-f$ & $0$ \\ \hline
      Preserves $\oplus$ by & Axiom \ref{scale axiom:mediality} & Axiom \ref{scale axiom:-(a+b)=(-a)+(-b)} & Axiom \ref{scale axiom:idempotence} \\
      Preserves $-$ by & Axiom \ref{scale axiom:-(a+b)=(-a)+(-b)} & Axiom \ref{scale axiom:-(-a)=a} & Proposition \ref{prop:-0=0} \\ \hline
    \end{tabular}
  \end{table}
\end{proof}

\section{Cancellative scales and their enveloping dy-modules}
\label{sec:cancellative scales and their enveloping dy-modules}

We have seen in Example \ref{eg:cancellative scales} (1) that every dy-module is a scale in an obvious way. Dy-modules and their subscales clearly have the following property: \begin{dfn}
  A scale $A$ is \demph{cancellative} if $a\oplus c=b\oplus c$ implies $a=b$ for all $a,b,c\in A$.
\end{dfn}

Conversely, every cancellative scale can be embedded into a dy-module. Freyd deduced that all Freyd scales are cancellative from the axioms (\cite[2.3 Lemma]{Fre08}), and constructed their \emph{enveloping dy-module} (\cite[Section 20]{Fre08}): briefly, given a Freyd scale $A$, we can take the underlying set of its enveloping dy-module to be a quotient set of $\set N\times A$ (where $\set N:=\{0,1,2,...\}$), so that a pair $(n,a)$ represents “$2^n$ times $a$.” This construction is also valid for general cancellative scales. Enveloping dy-modules are characterized with universality as follows: \begin{dfn}
  An \demph{enveloping dy-module} of a cancellative scale $A$ is a dy-module $E(A)$ together with a scale homomorphism $\varepsilon\colon A\to E(A)$ that has the following universality:\begin{quote}
    for every dy-module $B$ together with a scale homomorphism $f\colon A\to B$, there is a unique dy-module homomorphism $\overline f\colon E(A)\to B$ such that the triangle\[
      \begin{tikzcd}
        E(A) \arrow["\overline f", dashed, dr] \\
        A \arrow["f"', r] \arrow["\varepsilon", u] & B
      \end{tikzcd}
    \]commutes.
  \end{quote}
\end{dfn}

\begin{prop}
  \label{prop:homomorphisms between dy-modules}
  A map between dy-modules is a scale homomorphism if and only if it is a dy-module homomorphism.
\end{prop}
\begin{proof}
  Let $A$ and $B$ be dy-modules, and $f\colon A\to B$ be a map. As mentioned in Example \ref{eg:cancellative scales} (1), a dy-module is just an abelian group with an extra property. It suffices, therefore, to show that $f$ preserves midpoints if and only if it preserves addition. If $f$ preserves addition, then for all $a,b\in A$,\[
    2f(a\oplus b) = f(2(a\oplus b)) = f(a+b) = f(a)+f(b) = 2(f(a)\oplus f(b)),
  \]and thus $f(a\oplus b)=f(a)\oplus f(b)$. Conversely, if $f$ preserves midpoints, then for all $a,b\in A$,\begin{multline*}
    f(a+b)\oplus 0 = f(a+b)\oplus f(0) = f((a+b)\oplus 0) \\
    {} = f(a\oplus b) = f(a)\oplus f(b) = (f(a)+f(b))\oplus 0,
  \end{multline*}and thus $f(a+b)=f(a)+f(b)$. Hence we conclude that $f$ is a scale homomorphism if and only if it is a dy-module homomorphism.
\end{proof}

We write $\set D\categoryMod$ for the category whose objects are dy-modules and morphisms are dy-module homomorphism, and $\category{Scl}_c$ for the full subcategory of $\category{Scl}$ whose objects are cancellative scales. From the fact that every dy-module is a cancellative scale and Proposition \ref{prop:homomorphisms between dy-modules}, we obtain a fully faithful functor\[
  \begin{tikzcd}
    \set D\categoryMod \arrow[hook, r] & \category{Scl}_c,
  \end{tikzcd}
\]through which we can consider $\set D\categoryMod$ as a full subcategory of $\category{Scl}_c$. This functor has a left adjoint $E$, which sends each cancellative scale to its enveloping dy-module:\[
  \begin{tikzcd}
    \set D\categoryMod \arrow[hook, r, shift right=2] & \category{Scl}_c. \arrow["E"', "\bot", l, shift right=2]
  \end{tikzcd}
\]In other words, $\set D\categoryMod$ is a \emph{reflective subcategory} of $\category{Scl}_c$ with reflector $E$.

Of course, there are also \emph{non-cancellative} scales. Note that cancellativity is not written in pure equations, but in a Horn formula. In fact, the category $\category{Scl}_c$ is a \emph{quasivariety}, but not a variety: it is closed under taking subscales and direct products, but not under taking homomorphic images. (On quasivarieties, see \cite[Chapter V, Section 2]{BS12} for instance.) Here, we give an example of a non-cancellative scale which is a homomorphic image of a cancellative scale.

\begin{eg}
  \label{eg:three-element semilattice}
  The set $\{-1,0,1\}$ is a dy-module under addition modulo $3$, but there is another way to make this set into a scale. Define an order on $\{-1,0,1\}$ so that $0$ is the least, and $1$ and $-1$ are not comparable, as in the diagram
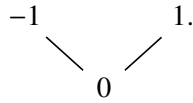
\[
    \begin{tikzcd}[sep=1em]
      -1 \arrow[dr, dash] & & 1. \arrow[dl, dash] \\
      & 0 & \phantom{-1}
    \end{tikzcd}
  \]Then $\{-1,0,1\}$ is a meet-semilattice, and furthermore, a scale when we define midpoints as meets, and negation and zero as they are. (In particular, the mediality follows from the commutativity and associativity of the meet operation.) This scale is not cancellative; for instance,\[
    0\wedge 0 = 0 = 1\wedge 0 \quad \text{but} \quad 0 \ne 1.
  \]Now, define a map $f\colon[-1,1]\to\{-1,0,1\}$ by\[
    f(a) := \begin{cases}
      a & \text{if }a=\pm 1; \\
      0 & \text{otherwise}.
    \end{cases}
  \]Then $f$ is a scale homomorphism; indeed, for all $a,b\in[-1,1]$,\begin{align*}
    f(a\oplus b) = f(a)\oplus f(b) &= \begin{cases}
      a & \text{if }(a,b)=(-1,-1)\text{ or }(a,b)=(1,1); \\
      0 & \text{otherwise};
    \end{cases} \\
    f(-a) = -f(a) &= \begin{cases}
      -a & \text{if }a=\pm 1; \\
      0 & \text{otherwise}.
    \end{cases}
  \end{align*}The scale $[-1,1]$ is cancellative, while its image $\{-1,0,1\}$ under $f$ is not cancellative. This shows that homomorphic images of cancellative scales are not necessarily cancellative.
\end{eg}

\section{Congruences and cancellativation}
\label{sec:congruences and cancellativation}

\begin{dfn}
  A \demph{congruence} on a scale $A$ is an equivalence relation $\sim$ on $A$ such that:\begin{itemize}
    \item (compatibility with $\oplus$) if $a\sim b$ and $c\sim d$, then $a\oplus c\sim b\oplus d$ for all $a,b,c,d\in A$;
    \item (compatibility with $-$) if $a\sim b$, then $(-a)\sim(-b)$ for all $a,b\in A$.
  \end{itemize}
\end{dfn}

\begin{dfn}
  Let $A$ be a scale, and $\sim$ be a congruence on it. The \demph{quotient scale} of $A$ by $\sim$ is the quotient set $A/{\sim}$ together with the fundamental operations defined so that the quotient map $A\twoheadrightarrow A/{\sim}$ is a scale homomorphism; that is, we define:\begin{itemize}
    \item $(a/{\sim})\oplus(b/{\sim}):=(a\oplus b)/{\sim}$ for all $a,b\in A$;
    \item $-(a/{\sim}):=(-a)/{\sim}$ for all $a\in A$;
    \item $0:=0/{\sim}$.
  \end{itemize}Here, $a/{\sim}$ denotes the equivalence class of $a$ under $\sim$. (The right-hand sides do not depend on the choice of the representative elements $a$ and $b$ by the compatibility of $\sim$ with $\oplus$ and $-$.)
\end{dfn}

In group theory, we rarely refer to congruences on a group, because they are in one-to-one correspondence with its normal subgroups (\cite[Theorem 1.21]{Ber11}). In particular, congruences on an abelian group are in one-to-one correspondence with its (arbitrary) subgroups, as follows. \begin{prop}
  \label{prop:congruences and submodules}
  Let $A$ be an abelian group. Then:\begin{enumerate}[label=(\arabic*)]
    \item for every congruence $\sim$ on $A$, the equivalence class $0/{\sim}$ is a subgroup of $A$;
    \item for every subgroup $S\subset A$, the relation $\equiv_{\mod S}$ on $A$ such that\[
      a \equiv_{\mod S} b \quad \text{if and only if} \quad a+(-b) \in S
    \]is a congruence;
    \item the two maps\[
      \begin{array}{cccc}
        & \{\text{subgroups of }A\} & \leftrightarrows & \{\text{congruences on }A\} \\
        & \rotatebox{90}{$\in$} & & \rotatebox{90}{$\in$} \\
        & 0/{\sim} & \reflectbox{$\mapsto$} & \sim \\
        \text{and} & S & \mapsto & \equiv_{\mod S}
      \end{array}
    \]are inverse of each other and preserve inclusion.
  \end{enumerate}
\end{prop}

Now, let us take $A$ to be a scale. The first question is:\begin{quote}
  are Proposition \ref{prop:congruences and submodules} (1) and (2) still true even when we replace “subgroup” with “subscale” and $+$ with $\oplus$?
\end{quote}The answer is yes for (1), while it is no for (2). For instance, take $A=[-2,2]\subset\set R$ and $S=[-1,1]$. Then the relation $\equiv$ on $A$ such that\[
  a\equiv b \quad \text{if and only if} \quad a\oplus(-b) \in S
\]is not transitive. Indeed:\begin{itemize}
  \item $2\equiv 0\equiv(-2)$ because $2\oplus(-0)=0\oplus(-(-2))=1\in S$;
  \item $2\not\equiv(-2)$ because $2\oplus(-(-2))=2\not\in S$.
\end{itemize}We can, however, also make (2) hold if we modify the definition of $\equiv_{\mod S}$.

\begin{ntn}
  Given an integer $n\ge 0$ and an element $a$ of a scale, we use the shorthand\[
    \frac{a}{2^n} := (\cdots((a\!\underbrace{{}\oplus 0)\oplus 0)\cdots\oplus 0)}_{n\text{ times}}.
  \]
\end{ntn}

\begin{prop}
  \label{prop:congruences and subscales}
  Let $A$ be a scale. Then:\begin{enumerate}[label=(\arabic*)]
    \item for every congruence $\sim$ on $A$, the equivalence class $0/{\sim}$ is a subscale of $A$;
    \item for every subscale $S\subset A$, then the relation $\equiv_{\mod S}$ on $A$ such that\[
      a \equiv_{\mod S} b \quad \text{if and only if} \quad \frac{a+(-b)}{2^n} \in S \text{ for some integer } n\ge 0
    \]is a congruence.
  \end{enumerate}
\end{prop}
\begin{proof}
  (1) The equivalence class $0/{\sim}$ is the inverse image of $\{0\}\subset A/{\sim}$ under the quotient map $A\twoheadrightarrow A/{\sim}$, which is a scale homomorphism. Since $\{0\}$ is a subscale of $A/{\sim}$, it follows that $0/{\sim}$ is a subscale of $A$.

  (2) For all $a,b,c,d\in A$:\begin{itemize}
    \item (reflexivity) $a\oplus(-a)=0\in S$;
    \item (symmetry) if $(a\oplus(-b))/2^n\in S$ for some integer $n\ge 0$, then\[
      \frac{b\oplus(-a)}{2^n} = -\frac{a\oplus(-b)}{2^n} \in S;
    \]
    \item (transitivity) if $(a\oplus(-b))/2^m,(b\oplus(-c))/2^n\in S$ for some integers $m,n\ge 0$, dividing one of them by two several more times we may assume that $m=n$, and then\[
      \frac{a\oplus(-c)}{2^{n+1}} = \frac{a\oplus(-b)}{2^n}\oplus\frac{b\oplus(-c)}{2^n} \in S;
    \]
    \item (compatibility with $\oplus$) if $(a\oplus(-b))/2^m,(c\oplus(-d))/2^n\in S$ for some integers $m,n\ge 0$, we may assume that $m=n$ similarly to above, and then\[
      \frac{(a\oplus c)\oplus(-(b\oplus d))}{2^n} = \frac{a\oplus(-b)}{2^n}\oplus\frac{c\oplus(-d)}{2^n} \in S;
    \]
    \item (compatibility with $-$) if $(a\oplus(-b))/2^n\in S$ for some integer $n\ge 0$, then\[
      \frac{(-a)\oplus(-(-b))}{2^n} = -\frac{a\oplus(-b)}{2^n} \in S.
    \]
  \end{itemize}Hence $\equiv_{\mod S}$ is a congruence on $A$.
\end{proof}

By Proposition \ref{prop:congruences and subscales}, we obtain two maps\begin{equation}
  \begin{array}{cccc}
    & \{\text{subscales of }A\} & \leftrightarrows & \{\text{congruences on }A\} \\
    & \rotatebox{90}{$\in$} & & \rotatebox{90}{$\in$} \\
    & 0/{\sim} & \reflectbox{$\mapsto$} & \sim \\
    \text{and} & S & \mapsto & \equiv_{\mod S},
  \end{array} \label{eq:congruences and subscales:one-to-one maps}
\end{equation}which obviously preserve inclusion by the definition of $0/{\sim}$ and $\equiv_{\mod S}$. The next question is:\begin{quote}
  are they inverses of each other?
\end{quote}The answer is again no, but it turns yes if we add some constraints on congruences and subscales.

\begin{dfn}
  Let $A$ be a scale.\begin{itemize}
    \item A congruence $\sim$ on $A$ is \demph{cancellative} if the quotient scale $A/{\sim}$ is cancellative, that is, $a\oplus c\sim b\oplus c$ implies $a\sim b$ for all $a,b,c\in A$.
    \item A subscale $S\subset A$ is \demph{doublable} if $a/2\in S$ implies $a\in S$ for all $a\in A$.
  \end{itemize}
\end{dfn}

\begin{prop}
  \label{prop:cancellative congruences and doublable subscales}
  Let $A$ be a scale. Then:\begin{enumerate}[label=(\arabic*)]
    \item for every cancellative congruence $\sim$ on $A$, the subscale $0/{\sim}\subset A$ is doublable;
    \item for every subscale $S\subset A$ (which is not necessarily doublable), the congruence $\equiv_{\mod S}$ on $A$ is cancellative.
  \end{enumerate}
\end{prop}
\begin{proof}
  (1) Let $a\in A$. Then $a\oplus 0\sim 0\oplus 0$ implies $a\sim 0$ by the cancellativity of $\sim$, which exactly means that $a/2\in 0/{\sim}$ implies $a\in 0/{\sim}$. Hence $0/{\sim}$ is doublable.

  (2) Let $a,b,c\in A$. If $a\oplus c\equiv_{\mod S}b\oplus c$, then there is an integer $n\ge 0$ such that\[
    S \ni \frac{(a\oplus c)\oplus(-(b\oplus c))}{2^n} = \frac{a\oplus(-b)}{2^{n+1}},
  \]and thus $a\equiv_{\mod S}b$. Hence $\equiv_{\mod S}$ is cancellative.
\end{proof}

\begin{prop}
  \label{prop:cancellative congruences and doublable subscales:inverses of each other}
  Let $A$ be a scale. Then:\begin{enumerate}[label=(\arabic*)]
    \item ${\sim}\subseteq{\equiv_{\mod 0/{\sim}}}$ for every congruence $\theta$ on $A$, with equality if $\theta$ is cancellative;
    \item $S\subseteq 0/{\equiv_{\mod S}}$ for every subscale $S\subset A$, with equality if $S$ is doublable.
  \end{enumerate}
\end{prop}
\begin{proof}
  (1) For all $a,b\in A$,\begin{align}
    a\sim b \ \ \Longrightarrow &\ \frac{a\oplus(-b)}{2^n}\sim\frac{b\oplus(-b)}{2^n}(=0)\text{ for some integer }n\ge 0 \label{eq:congruences and subscales:inverses of each other:01} \\
    \Longleftrightarrow &\ \frac{a\oplus(-b)}{2^n}\in 0/{\sim}\text{ for some integer }n\ge 0 \notag \\
    \Longleftrightarrow &\ a\equiv_{\mod 0/{\sim}}b, \notag
  \end{align}and the converse of \eqref{eq:congruences and subscales:inverses of each other:01} also holds if $\sim$ is cancellative. Hence ${\sim}\subseteq{\equiv_{\mod 0/{\sim}}}$, with equality if $\sim$ is cancellative.
  
  (2) For all $a\in A$,\begin{align}
    a\in S \ \ \Longrightarrow &\ \frac{a\oplus 0}{2^n}\in S\text{ for some integer }n\ge 0 \label{eq:congruences and subscales:inverses of each other:02} \\
    \Longleftrightarrow &\ a\equiv_{\mod S}0 \notag \\
    \Longleftrightarrow &\ a\in 0/{\equiv_{\mod S},} \notag
  \end{align}and the converse of \eqref{eq:congruences and subscales:inverses of each other:02} also holds if $S$ is doublable. Hence $S\subseteq 0/{\equiv_{\mod S}}$, with equality if $S$ is doublable.
\end{proof}

By Proposition \ref{prop:cancellative congruences and doublable subscales}, we can restrict the maps in \eqref{eq:congruences and subscales:one-to-one maps} to obtain two new maps\[
  \begin{array}{cccc}
    & \left\{\begin{gathered} \text{doublable} \\ \text{subscales of }A \end{gathered}\right\} & \leftrightarrows & \left\{\begin{gathered} \text{cancellative} \\ \text{congruences on }A \end{gathered}\right\} \\
    & \rotatebox{90}{$\in$} & & \rotatebox{90}{$\in$} \\
    & 0/{\sim} & \reflectbox{$\mapsto$} & \sim \\
    \text{and} & S & \mapsto & \equiv_{\mod S},
  \end{array}
\]which are inverses of each other by Proposition \ref{prop:cancellative congruences and doublable subscales:inverses of each other}. Furthermore, the latter can be factored as\[
  \left\{\begin{gathered} \text{doublable} \\ \text{subscales of }A \end{gathered}\right\} \hookrightarrow \{\text{subscales of }A\} \twoheadrightarrow \left\{\begin{gathered} \text{cancellative} \\ \text{congruences on }A \end{gathered}\right\}.
\]

Now, we explain how to make a “universal” cancellative scale from a general scale.

\begin{dfn}
  A \demph{cancellativation} of a scale $A$ is a cancellative scale $C(A)$ together with a scale homomorphism $\varepsilon\colon A\to C(A)$ that has the following universality:\begin{quote}
    for any cancellative scale $B$ together with a scale homomorphism $f\colon A\to B$, there is a unique scale homomorphism $\overline f\colon C(A)\to B$ such that the triangle\[
      \begin{tikzcd}
        A \arrow["f", r] \arrow["\varepsilon"', d] & B \\
        C(A) \arrow["\overline f"', dashed, ur]
      \end{tikzcd}
    \]commutes.
  \end{quote}
\end{dfn}

\begin{thm}
  Let $A$ be a scale. Then the quotient scale $A/{\equiv_{\mod\{0\}}}$, together with the quotient map $\varepsilon\colon A\twoheadrightarrow A/{\equiv_{\mod\{0\}}}$, is a cancellativation of $A$.
\end{thm}
\begin{proof}
  The quotient scale $A/{\equiv_{\mod\{0\}}}$ is cancellative by Proposition \ref{prop:cancellative congruences and doublable subscales} (2). To prove the universality, take any cancellative scale $B$ and a scale homomorphism $f\colon A\to B$. Our goal is to construct a unique scale homomorphism $\overline f\colon A/{\equiv_{\mod\{0\}}}\to B$ such that $\overline f(a/{\equiv_{\mod\{0\}}})=f(a)$ for all $a\in A$. The uniqueness of $\overline f$ follows from the surjectivity of the quotient map $\varepsilon$. To prove the existence of $\overline f$, it is necessary and sufficient to show that $a\equiv_{\mod\{0\}}b$ implies $f(a)=f(b)$ for all $a,b\in A$. If $a\equiv_{\mod\{0\}}b$, that is, there is an integer $n\ge 0$ such that\[
    \frac{a\oplus(-b)}{2^n} = 0,
  \]then\[
    \frac{f(a)\oplus(-f(b))}{2^n} = 0
  \]because $f$ is a scale homomorphism, and thus $f(a)=f(b)$ by the cancellativity of $B$. Hence we have obtained the desired scale homomorphism $\overline f\colon A/{\equiv_{\mod\{0\}}}\to B$.
\end{proof}

From a category-theoretic point of view, we can say that the inclusion functor from $\category{Scl}_c$ to $\category{Scl}$ has a left adjoint $C$, which sends each scale to its cancellativation:\[
  \begin{tikzcd}
    \category{Scl}_c \arrow[hook, r, shift right=2] & \category{Scl}. \arrow["\bot", "C"', l, shift right=2]
  \end{tikzcd}
\]In other words, $\bf{Scl}_c$ is a reflective subcategory of $\bf{Scl}$ with reflector $C$.

\section{Free scales}
\label{sec:free scales}

\begin{dfn}
  A \demph{free scale} on a set $J$ is a scale $F(J)$ together with a map $\varepsilon\colon J\to F(J)$ that has the following universality:\begin{quote}
    for every scale $A$ together with a map $f\colon J\to A$, there is a unique scale homomorphism $\overline f\colon F(J)\to A$ such that the triangle\[
      \begin{tikzcd}
        F(J) \arrow["\overline f", dashed, dr] \\
        J \arrow["f"', r] \arrow["\varepsilon", u] & A
      \end{tikzcd}
    \]commutes.
  \end{quote}If $J\subset F(J)$ and $\varepsilon$ is the inclusion map, then we also say that $F(J)$ is \demph{freely generated} by $J$.
\end{dfn}

Since the category $\category{Scl}$ is a variety, it is guaranteed that there exists a free scale on any set, and it is unique up to isomorphism. (For construction of free algebras, see \cite[Chapter II, Section 10]{BS12} for instance.) Then, is there any simple way to represent free scales? For free abelian groups, it is well known that:\begin{prop}
  \label{prop:free abelian groups}
  For each set $J$, the abelian group\[
    \set Z^{\oplus J} := \big\{(x_j)_{j\in J}\in\set Z^J\ \big|\ x_j\ne 0\text{ for only finitely many }j\in J\big\}
  \]is freely generated by $\{\varepsilon_j\mid j\in J\}$, where $\varepsilon_j$ is the point of which the $j$-coordinate is $1$ and all the others are $0$.
\end{prop}

If a subscale of $\set R$ contains a single element $1$, then it must contain\[
  0,\ -1,\ \pm\frac{1}{2} = \pm(0\oplus 1),\ \pm\frac{1}{4} = \pm\bigg(0\oplus\frac{1}{2}\bigg),\ \pm\frac{3}{4} = \pm\bigg(\frac{1}{2}\oplus 1\bigg),\ ...,
\]and thus, all dyadic rational numbers between $-1$ and $1$. Freyd used the symbol $\set I$ for the set of these numbers. Clearly $\set I$ is a subscale of $\set D$, and we can conjecture that the scale $\set I$ is freely generated by a single element $1$.

Is Proposition \ref{prop:free abelian groups} still true even when we replace “abelian group” with “scale” and $\set Z$ with $\set I$? If a subscale of $\set R^2$ contains two points $(1,0)$ and $(0,1)$, then it must contain\[
  (0,0),\ (-1,0),\ (0,-1),\ \bigg({\pm\frac{1}{2}},0\bigg),\ \bigg(0,\pm\frac{1}{2}\bigg),\ \bigg({\pm\frac{1}{2}},\pm\frac{1}{2}\bigg),\ \bigg({\pm\frac{1}{2}},\mp\frac{1}{2}\bigg),\ ...,
\]and thus, all points in the square defined by the inequality $|x|+|y|\le 1$ whose coordinates are both dyadic rational, as in Figure \ref{fig:l1-unit square}. To generalize this, we can conjecture that for each set $J$, the set\[
  F(J) := \Big\{(x_j)_{j\in J}\in\set I^J\ \Big|\ x_j\ne 0\text{ for only finitely many }j\in J\text{ and }\sum_{j\in J}|x_j|\le 1\Big\}
\]is a subscale of $\set I^J$, and this scale is freely generated by $\{\varepsilon_j\mid j\in J\}$ (where $\varepsilon_j$ is the same as in Proposition \ref{prop:free abelian groups}). Here, we should note that there is an extra condition\[
  \sum_{j\in J}|x_j| \le 1,
\]which was absent for free abelian groups. The left-hand side is known as the \emph{$L^1$-norm} of the point $(x_j)_{j\in J}$ (\cite[Chapter III, 1.9 Example]{Con07} with $p=1$). By the definition of norms, it is easy to see that $F(J)$ is a subscale of $\set I^J$. \begin{figure}
  \centering
  \begin{tikzpicture}
      \filldraw[fill=lightgray] (1,0)--(0,1)--(-1,0)--(0,-1)--cycle;
      \draw[->] (-1.5,0)--(1.5,0) node[right]{$x$};
      \draw[->] (0,-1.5)--(0,1.5) node[above]{$y$};
      \fill (-1,0) circle(2pt) node[below left]{$-1$};
      \fill (-0.5,-0.5) circle(2pt);
      \fill (-0.5,0) circle(2pt);
      \fill (-0.5,0.5) circle(2pt);
      \fill (0,-1) circle(2pt) node[below left]{$-1$};
      \fill (0,-0.5) circle(2pt);
      \fill (0,0) circle(2pt) node[below left]{$\mathrm O$};
      \fill (0,0.5) circle(2pt);
      \fill (0,1) circle(2pt) node[above left]{$1$};
      \fill (0.5,-0.5) circle(2pt);
      \fill (0.5,0) circle(2pt);
      \fill (0.5,0.5) circle(2pt);
      \fill (1,0) circle(2pt) node[below right]{$1$};
    \end{tikzpicture}
  \caption{The square in $\set R^2$ defined by the inequality $|x|+|y|\le 1$.}
  \label{fig:l1-unit square}
\end{figure}
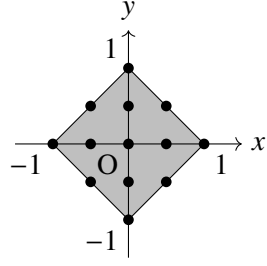

How should we prove that the scale $F(J)$ is free? Using summation, we can express each point $x=(x_j)_{j\in J}\in F(J)$ also as\[
  x = \sum_{j\in J}x_j\varepsilon_j.
\]Thus, given a scale $A$ and a map $f\colon J\to A$, we wish to define a map $\overline f\colon F(J)\to A$ by\[
  \overline f(x) := \text{“}\sum_{j\in J}x_jf(j).\text{”}
\]If $A$ is cancellative, then we can take in its enveloping dy-module the sum in the right-hand side, and it proves to be again in $A$. If $A$ is not cancellative, however, then the sum does not make sense. To solve this problem, we introduce \emph{formal sums} and their quotients by powers of two.

\begin{dfn}
  Let $A$ be a set, and fix a free commutative monoid on $A$. When we refer to each $a\in A$ as a generator, we write $[a]$ for it. The \demph{formal sum} of $n$ elements $a_1,a_2,...,a_n\in A$ is the sum $[a_1]+[a_2]+\cdots+[a_n]$ in that monoid.
\end{dfn}

\begin{dfn}
  \label{dfn:quotients of formal sums}
  Let $A$ be a scale, $n\ge 0$ be an integer, and $S$ be a formal sum of $2^n$ or less elements of $A$ (which are not necessarily distinct). We define $S/2^n$ recursively as follows:
  \begin{enumerate}[label=(\arabic*)]
    \item for $n=0$, we define\[
      \frac{S}{2^0} := \begin{cases}
        0 & \text{if $S=0$ (the empty sum)}; \\
        a & \text{if $S=[a]$ for some $a\in A$};
      \end{cases}
    \]
    \item for $n\ge 1$, splitting $S$ into two formal sums $S_1$ and $S_2$ of $2^{n-1}$ or less elements, we define\begin{equation}
      \frac{S}{2^n} := \frac{S_1}{2^{n-1}}\oplus\frac{S_2}{2^{n-1}}. \label{eq:quotients of formal sums:recursive step}
    \end{equation}
  \end{enumerate}
\end{dfn}

Of course, we have to check that the right-hand side of \eqref{eq:quotients of formal sums:recursive step} does not depend on the choice of $S_1$ and $S_2$. To do so, we use the following lemma: \begin{lem}
  \label{lem:quotients of formal sums:filling 0}
  Assume that division by $2^{n-1}$ is already defined. Let $T$ be a formal sum of $l$ elements of $A$, and $m$ be an integer such that $0\le m\le 2^{n-1}-l$. Then\[
    \frac{T}{2^{n-1}} = \frac{T+m[0]}{2^{n-1}}.
  \]
\end{lem}
\begin{proof}
  For $n=1$, the result is obvious by the fact that $0/2^0=0=[0]/2^0$. For $n\ge 2$, assume that it holds for one smaller $n$. If $m\ge 2^{n-2}$, then we have\[
    \frac{T+m[0]}{2^{n-1}} = \frac{T+(m-2^{n-2})[0]}{2^{n-2}}\oplus\frac{0+2^{n-2}[0]}{2^{n-2}} = \frac{T}{2^{n-2}}\oplus\frac{0}{2^{n-2}} = \frac{T}{2^{n-1}};
  \]if $m<2^{n-2}$, then splitting $T$ into a formal sum $T_1$ of $(l+m-2^{n-2})$ elements and $T_2$ of $(2^{n-2}-m)$ elements, we have\[
    \frac{T+m[0]}{2^{n-1}} = \frac{T_1}{2^{n-2}}\oplus\frac{T_2+m[0]}{2^{n-2}} = \frac{T_1}{2^{n-2}}\oplus\frac{T_2}{2^{n-2}} = \frac{T}{2^{n-1}}.
  \](In both cases, the first and third equalities are by the definition of division by $2^{n-1}$, and the second is by the induction hypothesis.) Hence, the result holds for all $n\ge 1$.
\end{proof}

\begin{prop}
  In the setting of Definition \ref{dfn:quotients of formal sums} (2), split $S$ into two other formal sums $S'_1$ and $S'_2$ of $2^{n-1}$ or less elements. Then\[
    \frac{S_1}{2^{n-1}}\oplus\frac{S_2}{2^{n-1}} = \frac{S'_1}{2^{n-1}}\oplus\frac{S'_2}{2^{n-1}}.
  \]
\end{prop}
\begin{proof}
  Here, we consider \emph{$L^1$-distance} between formal sums; given two formal sums $T_1$ and $T_2$ of elements of $A$, we define\[
    d_1(T_1,T_2) := \sum_{a\in A}|m^{(1)}_a-m^{(2)}_a|,
  \]where $m^{(i)}_a$ is the number of summands $[a]$ in $T_i$ for $i=1,2$. By Lemma \ref{lem:quotients of formal sums:filling 0}, for $i=1,2$, we may add some $[0]$ to either $S_i$ or $S'_i$ so that they have the same number of summands. Then it is easy to see that $d_1(S_1,S'_1)=d_1(S_2,S'_2)$, and this $L^1$-distance is even. Since $S_1\ne S'_1$, there are two distinct $a_1,a_2\in A$ such that $S'_1$ has less summands $[a_1]$ and more summands $[a_2]$ than $S_1$. Then\[
    d_1(S_1,\,S_1-[a_1]+[a_2]) = 2 \quad \text{and} \quad d_1(S_1-[a_1]+[a_2],\,S'_1) = d(S_1,S'_1)-2,
  \]as in Figure \ref{fig:positional relation among three points}. Thus, by induction, it suffices to consider the case where $d(S_1,S'_1)=2$. In this case\[
    S'_1 = S_1-[a_1]+[a_2] \quad \text{and} \quad S'_2 = S_2+[a_1]-[a_2].
  \]For $i=1,2$, split $S_i$ into two formal sums $S_{i1}$ and $S_{i2}$ of $2^{n-2}$ or less elements so that $S_{i2}$ has at least one summand $[a_i]$. Then\begin{align}
    & \frac{S_1-[a_1]+[a_2]}{2^{n-1}}\oplus\frac{S_2+[a_1]-[a_2]}{2^{n-1}} \notag \\
    &= \bigg(\frac{S_{11}}{2^{n-2}}\oplus\frac{S_{12}-[a_1]+[a_2]}{2^{n-2}}\bigg)\oplus\bigg(\frac{S_{21}}{2^{n-2}}\oplus\frac{S_{22}+[a_1]-[a_2]}{2^{n-2}}\bigg) \label{eq:quotients of formal sums:well-definedness:01} \\
    &= \bigg(\frac{S_{11}}{2^{n-2}}\oplus\frac{S_{21}}{2^{n-2}}\bigg)\oplus\bigg(\frac{S_{12}-[a_1]+[a_2]}{2^{n-2}}\oplus\frac{S_{22}+[a_1]-[a_2]}{2^{n-2}}\bigg) \label{eq:quotients of formal sums:well-definedness:02} \\
    &= \frac{S_{11}+S_{21}}{2^{n-1}}\oplus\frac{S_{12}+S_{22}}{2^{n-1}}, \label{eq:quotients of formal sums:well-definedness:03}
  \end{align}where \eqref{eq:quotients of formal sums:well-definedness:01} and \eqref{eq:quotients of formal sums:well-definedness:03} are by the definition of division by $2^{n-1}$, and \eqref{eq:quotients of formal sums:well-definedness:02} is by the mediality. Since these equalities hold even when we remove “${}-[a_1]+[a_2]$” and “${}+[a_1]-[a_2]$,” the result follows. 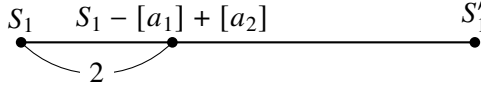
\begin{figure}[h]
    \centering
    \begin{tikzpicture}
      \draw[thick] (-2,0)--(4,0);
      \draw (-2,0) to[out=-45,in=225] (0,0);
      \draw (-1,-0.4) node[fill=white]{$2$};
      \fill (-2,0) circle(2pt) node[above]{$S_1$};
      \fill (0,0) circle(2pt) node[above]{$S_1-[a_1]+[a_2]$};
      \fill (4,0) circle(2pt) node[above]{$S'_1$};
    \end{tikzpicture}
    \caption{The positional relation among $S_1$, $S_1-[a_1]+[a_2]$, and $S'_1$.}
    \label{fig:positional relation among three points}
  \end{figure}
\end{proof}

Thus, we have checked that Definition \ref{dfn:quotients of formal sums} is valid. Using induction on $n$, we can also prove the following lemma easily.

\begin{lem}
  \label{lem:properties of formal sums}
  Let $A$ be a scale, $n\ge 0$ be an integer, and $a,a_1,a_2,...,a_k\in A$, where $k\le 2^n$. Then:\begin{enumerate}[label=(\arabic*)]
    \item $\displaystyle \frac{1}{2^n}\sum_{i=1}^ka_i=\frac{1}{2^n}\sum_{i=1}^k[a_i]$ if $A$ is cancellative \\
    (where the sum in the left-hand side is an actual one taken in the enveloping dy-module of $A$, while that in the right-hand side is a formal one); \label{lem:actual and formal sums coincide}
    \item $\displaystyle \frac{1}{2^{n+1}}\sum_{i=1}^k2[a_i]=\frac{1}{2^n}\sum_{i=1}^k[a_i]$; \label{lem:duplicating formal sums}
    \item $\displaystyle \frac{1}{2^n}\Big([a]+[-a]+\sum_{i=1}^k[a_i]\Big) = \frac{1}{2^n}\Big(2[0]+\sum_{i=1}^k[a_i]\Big)$ if $k\le 2^n-2$. \label{lem:liquidation of formal sums}
  \end{enumerate}Furthermore, let $B$ be a scale, and $f\colon A\to B$ be a scale homomorphism. Then:\begin{enumerate}[label=(\arabic*), resume]
    \item $\displaystyle f\bigg(\frac{1}{2^n}\sum_{i=1}^k[a_i]\bigg)=\frac{1}{2^n}\sum_{i=1}^k[f(a_i)]$. \label{lem:homomorphisms preserve formal sums}
  \end{enumerate}
\end{lem}

Now, we are ready for the proof of the main result: \begin{thm}
  \label{thm:free scales}
  For each set $J$, the scale\[
    F(J) := \Big\{(x_j)_{j\in J}\in\set I^J\ \Big|\ x_j\ne 0\text{ for only finitely many }j\in J\text{ and }\sum_{j\in J}|x_j|\le 1\Big\}
  \]is freely generated by $\{\varepsilon_j\mid j\in J\}$, where $\varepsilon_j$ is the point of which the $j$-coordinate is $1$ and all the others are $0$.
\end{thm}
\begin{proof}
  Let $A$ be a scale, and $f\colon J\to A$ be a map. First, we investigate if there were a scale homo\-morphism $\overline f\colon F(J)\to A$ such that $\overline f(\varepsilon_j)=f(j)$ for all $j\in J$, then what $\overline f(x)$ would have to be for each $x=(x_j)_{j\in J}\in F(J)$. Choosing an integer $n(\ge 0)$ so that $2^nx_j$ is an integer for all $j\in J$, we can express\[
    x = \frac{1}{2^n}\sum_{j\in J}2^n|x_j|\sgn(x_j)\varepsilon_j.
  \](Here, $\sgn(x_j)$ denotes the sign of $x_j$, which is $1$, $0$, or $-1$.) Note that $\sum_{j\in J}2^n|x_j|\le 2^n$ since $\sum_{j\in J}|x_j|\le 1$. Thus, applying Lemma \ref{lem:properties of formal sums} \ref{lem:actual and formal sums coincide} and \ref{lem:homomorphisms preserve formal sums} in this order, we must have\[
    \overline f(x) = \frac{1}{2^n}\sum_{j\in J}2^n|x_j|[\sgn(x_j)f(j)].
  \](Interpret $1f(j):=f(j)$, $0f(j):=0$, and $(-1)f(j):=-f(j)$.) Conversely, we define a map $\overline f\colon F(J)\to A$ in this way. Using Lemma \ref{lem:properties of formal sums} \ref{lem:duplicating formal sums} repeatedly, we find that the right-hand side does not depend on the choice of $n$. Since negation $-\colon A\to A$ is a scale homomorphism, it immediately follows from Lemma \ref{lem:properties of formal sums} \ref{lem:homomorphisms preserve formal sums} that $\overline f$ preserves negation. To check that $\overline f$ preserves midpoints, take any\[
    x = (x_j)_{j\in J}, \ y = (y_j)_{j\in J} \in F(J),
  \]and choose an integer $n$ so that $2^nx_j$ and $2^ny_j$ are integers for all $j\in J$. As in Figure \ref{fig:partition into six domains}, partition the plane $\set R^2$ into the following six domains:\begin{align*}
    D_1 &:= \{(x,y)\in\set R^2\mid x\ge 0\text{ and }y\ge 0\}; \\
    D_2 &:= \{(x,y)\in\set R^2\mid x<0\text{ and }y<0\}; \\
    D_3 &:= \{(x,y)\in\set R^2\mid x<0,\ y>0,\text{ and }x+y\ge 0\}; \\
    D_4 &:= \{(x,y)\in\set R^2\mid x\ge 0,\ y<0,\text{ and }x+y<0\}; \\
    D_5 &:= \{(x,y)\in\set R^2\mid x>0,\ y<0,\text{ and }x+y\ge 0\}; \\
    D_6 &:= \{(x,y)\in\set R^2\mid x<0,\ y\ge 0,\text{ and }x+y<0\}.
  \end{align*}Let $J_k:=\{j\in J\mid(x_j,y_j)\in D_k\}$ for $k=1,2,...,6$. Then\begin{align}
    \overline f(x)\oplus\overline f(y) &= \frac{1}{2^n}\sum_{j\in J}2^n|x_j|[\sgn(x_j)f(j)]\oplus\frac{1}{2^n}\sum_{j\in J}2^n|y_j|[\sgn(y_j)f(j)] \label{eq:free scales:f preserves midpoints:01} \\
    &= \frac{1}{2^{n+1}}\sum_{j\in J}(2^n|x_j|[\sgn(x_j)f(j)]+2^n|y_j|[\sgn(y_j)f(j)]) \label{eq:free scales:f preserves midpoints:02} \\
    &= \frac{1}{2^{n+1}}\Big(\sum_{j\in J}2^{n+1}|x_j\oplus y_j|[\sgn(x_j\oplus y_j)f(j)] \label{eq:free scales:f preserves midpoints:03} \\
    &\hspace{4em} {}+\hspace{-0.5em}\sum_{j\in J_3\cup I_4}\hspace{-0.5em}2^n|x_j|([f(j)]+[-f(j)])+\hspace{-0.5em}\sum_{j\in J_5\cup I_6}\hspace{-0.5em}2^n|y_j|([f(j)]+[-f(j)])\Big) \notag \\
    &= \frac{1}{2^{n+1}}\sum_{j\in J}2^{n+1}|x_j\oplus y_j|[\sgn(x_j\oplus y_j)f(j)] \label{eq:free scales:f preserves midpoints:04} \\
    &= \overline f(x\oplus y), \label{eq:free scales:f preserves midpoints:05}
  \end{align}where \eqref{eq:free scales:f preserves midpoints:01} is by the definition of $\overline f$, \eqref{eq:free scales:f preserves midpoints:02} is by the definition of division by $2^{n+1}$, \eqref{eq:free scales:f preserves midpoints:03} is by an elementary (yet a little bit tedious) calculation, \eqref{eq:free scales:f preserves midpoints:04} is by Lemma \ref{lem:properties of formal sums} \ref{lem:liquidation of formal sums}, and \eqref{eq:free scales:f preserves midpoints:05} is, again, by the definition of $\overline f$. Thus, we have obtained the desired scale homomorphism $\overline f\colon F(J)\to A$. 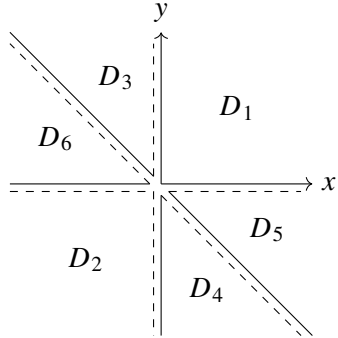
\begin{figure}[h]
    \centering
    \begin{tikzpicture}
      \draw (1,1) node {$D_1$};
      \draw[->] (0,0)--(2,0) node[right] {$x$};
      \draw[->] (0,0)--(0,2) node[above] {$y$};
      \draw (-1,-1) node {$D_2$};
      \draw[dashed] (-2,-0.1)--(-0.1,-0.1)--(-0.1,-2);
      \draw (-0.6,1.4) node {$D_3$};
      \draw (-0.1,0.1)--(-2,2);
      \draw[dashed] (-0.1,0.1)--(-0.1,1.85);
      \draw (0.6,-1.4) node {$D_4$};
      \draw (0,-0.15)--(0,-2);
      \draw[dashed] (0,-0.15)--(1.85,-2);
      \draw (1.4,-0.6) node {$D_5$};
      \draw (0.1,-0.1)--(2,-2);
      \draw[dashed] (0.1,-0.1)--(1.85,-0.1);
      \draw (-1.4,0.6) node {$D_6$};
      \draw (-0.15,0)--(-2,0);
      \draw[dashed] (-0.15,0)--(-2,1.85);
    \end{tikzpicture}
    \caption{Partition of $\set R^2$ into six domains.}
    \label{fig:partition into six domains}
  \end{figure}
\end{proof}

\begin{rmk} \
  \begin{enumerate}[label=(\arabic*)]
    \item Freyd proved that the Freyd scale of all “continuous piecewise dy-affine” functions from $\closedinterval^n$ to $\closedinterval$ (where $\closedinterval$ is any closed interval) is freely generated by the $n$ projections (\cite[Theorem 20.1]{Fre08}). Letting it denoted by $A$, and considering it as a symmetric midpoint algebra, we have an isomorphism\[
      F(\{1,2,...,n+1\}) \xrightarrow{\sim} A
    \]which sends $\varepsilon_1,\varepsilon_2,...,\varepsilon_n$ to the $n$ projections, and $\varepsilon_{n+1}$ to the constant function $\top$ whose value is the top of the interval $\closedinterval$. When we consider $A$ as a Freyd scale, the function $\top$ is a fundamental operation, and thus not counted as a generator. This is why the number of generators differs depending on what type of algebra we consider $A$ as.
    \item There are two possible reasons why Freyd used the symbol $\set I$ for the set of dyadic rational numbers between $-1$ and $1$. One is simply that it is in the \emph{i}nterval $[-1,1]$. The other is that, when we consider $\set I$ as a Freyd scale under the fundamental operations inherited from $[-1,1]$ (which was defined in Example \ref{eg:Freyd scales}), it is an \emph{i}nitial object of the category of Freyd scales (or precisely, the category whose objects are Freyd scales and morphisms are maps preserving all the fundamental operations $\mid$, $^\bullet$, $^\wedge$, and $\top$) (\cite[3.1 Theorem]{Fre08}). Note that the initial object of our category $\category{Scl}$ is not $F(\{1\})=\set I$, but $F(\varnothing)=\{0\}$.
  \end{enumerate}
\end{rmk}

From a category-theoretic point of view, we can say that the forgetful functor $U\colon\category{Scl}\to\category{Set}$ has a left adjoint $F$, which sends each set to a free scale on it:\[
  \begin{tikzcd}
    \category{Scl} \arrow["U"', r, shift right=2] & \category{Set}. \arrow["\bot", "F"', l, shift right=2]
  \end{tikzcd}
\]

\begin{cor}
  The category $\category{Scl}$ is the smallest variety that contains $\set I$.
\end{cor}
\begin{proof}
  We already know that $\category{Scl}$ is a variety. Every scale $A$ is a homomorphic image of a free scale; more precisely, $A$ is the image of the unique scale homomorphism $\overline{\mathrm{id}}\colon F(A)\to A$ such that the triangle\[
    \begin{tikzcd}
      F(A) \arrow["\overline{\mathrm{id}}", dr] \\
      A \arrow["\mathrm{id}"', r] \arrow["\varepsilon", u] & A
    \end{tikzcd}
  \](where $\mathrm{id}$ denotes the identity map on $A$) commutes. Furthermore, by Theorem \ref{thm:free scales}, every free scale is a subscale of a direct product of copies of $\set I$. Hence, if a variety of algebras of type $\langle 2,1,0\rangle$ contains $\set I$, then it must contain all scales.
\end{proof}

\section{Tensor products and scalic rings}

\begin{dfn}
  Let $A$, $B$, and $C$ be scales. A map $f\colon A\times B\to C$ is a \demph{bihomomorphism} if:\begin{itemize}
    \item for every $a\in A$, the map $B\ni b\mapsto f(a,b)\in C$ is a scale homomorphism;
    \item for every $b\in B$, the map $A\ni a\mapsto f(a,b)\in C$ is a scale homomorphism.
  \end{itemize}
\end{dfn}

\begin{dfn}
  A \demph{tensor product} of scales $A$ and $B$ is a scale $A\otimes B$ together with a bihomomorphism $\otimes\colon A\times B\to A\otimes B$ that has the following universality:\begin{quote}
    for every scale $C$ together with a bihomomorphism $f\colon A\times B\to C$, there is a unique scale homomorphism $\overline f\colon A\otimes B\to C$ such that the triangle\[
      \begin{tikzcd}
        A\otimes B \arrow["\overline f", dashed, dr] \\
        A\times B \arrow["f"', r] \arrow["\otimes", u] & C
      \end{tikzcd}
    \]commutes.
  \end{quote}
\end{dfn}

Tensor products of scales can be explicitly constructed in almost the same way as how to construct those of abelian groups, or more generally, modules over a commutative ring (which can be found in \cite[Chapter XVI, Section 1]{Lan02}). The only difference is using a congruence instead of a subalgebra to take a quotient. Given scales $A$ and $B$, let $\sim$ be the smallest congruence on the free scale $F(A\times B)$ such that for all $a,a'\in A$ and $b,b'\in B$:\begin{itemize}
  \item $\varepsilon(a,b\oplus b')\sim\varepsilon(a,b)\oplus\varepsilon(a,b')$;
  \item $\varepsilon(a\oplus a',b)\sim\varepsilon(a,b)\oplus\varepsilon(a',b)$;
  \item $\varepsilon(a,-b)\sim{-\varepsilon(a,b)}\sim\varepsilon(-a,b)$.
\end{itemize}(Here, $\varepsilon\colon A\times B\hookrightarrow F(A\times B)$ is the canonical injection). Then we can take $A\otimes B:=F(A\times B)/{\sim}$.

The category $\category{Ab}$ of abelian groups is a \emph{symmetric monoidal closed category} with respect to the usual tensor products, with monoidal unit $\set Z$, and rings can be defined as monoid objects in $\category{Ab}$ (\cite[Chapter VII]{ML98}). The category $\category{Scl}$ is also a symmetric monoidal closed category in exactly the same way, with monoidal unit $F(\{1\})\cong\set I$. Thus, let us call monoid objects in $\category{Scl}$ \emph{scalic rings}. Explicitly: \begin{dfn}
  A \demph{scalic ring} is an algebra $(A,\oplus,\cdot,-,0,1)$ of type $\langle 2,2,1,0,0\rangle$, consisting of:\begin{itemize}
    \item operations $\oplus$, $-$, and $0$ under which $A$ is a scale;
    \item a binary operation $\cdot$ called \demph{multiplication};
    \item a nullary operation $1$ called the \demph{unity},
  \end{itemize}satisfying the following axioms for all $a,b,c\in A$:\begin{itemize}
    \item (associativity) $(a\cdot b)\cdot c=a\cdot(b\cdot c)$;
    \item (unitality) $a\cdot 1=a=1\cdot a$;
    \item (distributivity) $\begin{cases}
      a\cdot(b\oplus c) = a\cdot b\oplus a\cdot c; \\
      (a\oplus b)\cdot c = a\cdot c\oplus b\cdot c; \\
      a\cdot(-b) = -(a\cdot b) = (-a)\cdot b.
    \end{cases}$
  \end{itemize}
\end{dfn}

\begin{dfn}
  A \demph{scalic subring} of a scalic ring $A$ is a subscale $S\subset A$ such that:\begin{itemize}
    \item (closure under multiplication) $ab\in S$ for all $a,b\in S$;
    \item $1\in S$.
  \end{itemize}
\end{dfn}

\begin{eg} \
  \begin{enumerate}[label=(\arabic*)]
    \item Like every module over the ring $\set D$ of dyadic rational numbers (which we have so far called dy-module) is a scale, every (associative unital) algebra over $\set D$ is a scalic ring in an obvious way.
    \item Let $A$ be a normed algebra (where $\|1\|=1$) over $\set R$ or $\set C$, and $B$ be a subalgebra of $A$ over $\set D$. Then the set $\{x\in B\mid\|x\|\le 1\}$, that is, the intersection of $B$ and the closed unit ball of $A$ is a scalic subring of $B$. In particular:\begin{itemize}
      \item the set $\set I$ of dyadic rational numbers between $-1$ and $1$ is a scalic subring of $\set D$;
      \item the closed interval $[-1,1]$ is a scalic subring of $\set R$;
      \item the closed unit disk $\{z\in\set C\mid|z|\le 1\}$ is a scalic subring of $\set C$.
    \end{itemize}
    \item Let $L$ be a distributive lattice with a greatest element $\top$. (We do not require $L$ to have a least element.) Define\[
      L' := (\{-1,0,1\}\times L)/{\sim},
    \]where $\sim$ is the equivalence relation such that for all $s,t\in\{-1,0,1\}$ and $a,b\in L$,\[
      (s,a)\sim(t,b) \quad \text{if and only if} \quad (s,a)=(t,b)\text{ or }s=t=0.
    \]Then $L'$ is a scalic ring when we define:\begin{align*}
      (s,a)\oplus(t,b) &:= (s\barwedge t,\,a\vee b); \\
      -(s,a) &:= (-s,a); \\
      0 &:= (0,\top); \\
      (s,a)\cdot(t,b) &:= (s\cdot t,\,a\wedge b); \\
      1 &:= (1,\top).
    \end{align*}Here, $s\barwedge t$ denotes the meet of $s$ and $t$ defined in Example \ref{eg:three-element semilattice}, and $s\cdot t$ their usual product. By abuse of notation, we write $(s,a)$ to mean the equivalence class of the pair $(s,a)$. Since $0$ is an absorbing element for both $\barwedge$ and $\cdot$\,, even if $s=0$ or $t=0$, the right-hand sides do not depend on the choice of $a$ and $b$. Of course, the distributivity\[
      (\alpha\oplus\beta)\cdot\gamma = \alpha\cdot\gamma\oplus\beta\cdot\gamma \quad \text{and} \quad \gamma\cdot(\alpha\oplus\beta) = \gamma\cdot\alpha\oplus\gamma\cdot\beta
    \](where $\alpha,\beta,\gamma\in L'$) follows from that of the lattice $L$ if $\alpha\oplus\beta\ne 0$; otherwise these equations still hold because both sides will be $0$ by the definitions of $\oplus$ and $\cdot$\,.
    \item Let $M$ be a cancellative monoid, that is, a monoid such that for all $a,b,c\in M$,\[
      ab=ac\text{ implies }b=c, \quad \text{and} \quad ac=bc\text{ implies }a=b.
    \]Define a set $M'$ in exactly the same way as how we defined $L'$ in (3). Then $M'$ is a scalic ring when we define the fundamental operations similarly, with modification as follows:\begin{itemize}
      \item replace $\wedge$ and $\top$ with the binary operation and identity element of $M$, respectively;
      \item if $s=t$, then redefine\[
        (s,a)\oplus(t,b) := \begin{cases}
          (s,a) & \text{if }a=b; \\
          0 & \text{otherwise}.
        \end{cases}
      \]
    \end{itemize}In fact, the distributivity\[
      (\alpha\oplus\beta)\cdot\gamma = \alpha\cdot\gamma\oplus\beta\cdot\gamma \quad \text{and} \quad \gamma\cdot(\alpha\oplus\beta) = \gamma\cdot\alpha\oplus\gamma\cdot\beta
    \](where $\alpha,\beta,\gamma\in M'$) relies largely on the cancellativity of the monoid $M$: if $\alpha\ne\beta$, then it follows that $\alpha\cdot\gamma\ne\beta\cdot\gamma$, so that both sides will be $0$; otherwise these equations still hold by the idempotence of $\oplus$.
  \end{enumerate}
\end{eg}

\begin{rmk}
  The theories of abelian groups and scales are both \emph{commutative theories}, which means that every fundamental operation commutes with every fundamental operation. For instance, the mediality of midpointing $\oplus$ exactly means that $\oplus$ commutes with itself. In fact, every variety of algebras with commutative theory admits tensor products to become a symmetric monoidal closed category, similarly to $\category{Ab}$ and $\category{Scl}$ (\cite[Theorem 3.10.3]{Bor94}).
\end{rmk}

\end{document}